\documentclass[12pt]{amsart}

\usepackage{amsfonts}
\usepackage{caption}
\usepackage{forest}
\usepackage{amsmath}
\usepackage{amssymb}
\usepackage{tikz}
\usepackage[cp850]{inputenc}
\usepackage[bookmarksnumbered,plainpages
]{hyperref}
\usepackage{color}
\usepackage{mathtools}
\usepackage{MnSymbol}
\usepackage{subfigure}
\usepackage[all]{xy}

\allowdisplaybreaks
\newtheorem{theorem}{Theorem}[section]
\newtheorem*{theorem*}{Theorem}

\newtheorem{proposition}[theorem]{Proposition}
\newtheorem{corollary}[theorem]{Corollary}

\newtheorem{lemma}[theorem]{Lemma}

\theoremstyle{definition}
\newtheorem{example}[theorem]{Example}

\def\sym#1{\mathrm{Sym}(#1)}
\def\c#1{\mathrm{con}_{#1}}

\def\aut#1{\mathrm{Aut}(#1)}

\def\aff#1{\mathrm{Aff}#1}

\def\lmlt{\mathrm{LMlt}}
\def\dis{\mathrm{Dis}}
\def\Q{\mathcal{Q}}%_{\mathrm{Hom}}}

\def\N{\mathrm{Norm}}
\newcommand{\Con}{\mathrm{Con}}

\def\setof#1#2{\{#1\, : \,#2\}}

\newcommand*\xbar[1]{%
	\hbox{%
		\vbox{%
			\hrule height 0.5pt % The actual bar
			\kern0.5ex%         % Distance between bar and symbol
			\hbox{%
				\kern-0.1em%      % Shortening on the left side
				\ensuremath{#1}%
				\kern-0.1em%      % Shortening on the right side
			}%
		}%
	}%
}

\def\ldiv{\backslash}

\def\!#1{#1^\s}

\def\s{\mathfrak{s}}

\makeindex 
\title{Mal'cev classes of left quasigroups and Quandles}

\author{M. Bonatto}

\address[M. Bonatto]{Dipartimento di matematica e informatica - UNIFE}

\email{marco.bonatto.87@gmail.com}

\author{S. Fioravanti}

\address[S. Fioravanti]{Institute for Algebra - JKU Linz}

\email{stefano.fioravanti.66@gmail.com}

\begin{document}
	\maketitle
%\tableofcontents

\section*{Abstract}
In this paper we investigate some Mal'cev classes of varieties of left quasigroups. We prove that the weakest non-trivial Mal'cev condition for a variety of left quasigroups is having a Mal'cev term and that every congruence meet-semidistributive variety of left quasigroups is congruence arithmetic. Then we specialize to the setting of quandles for which we prove that the congruence distributive varieties are those which has  finite models.

\section{Introduction}

Starting from Mal'cev's description of congruence permutability as in \cite{Malt}, the problem of characterizing properties of classes of varieties as {\it Mal'cev conditions} has led to several results. Mal'cev conditions turned out to be extremely useful, for instance to capture lattice theoretical properties of the congruence lattices of the algebras of classes of variety. In \cite{Pix.DAPO} A. Pixley found a strong Mal'cev condition defining the class of varieties with distributive and permuting congruences. In \cite{Jon.AWCL} B. J\'{o}nsson shows a Mal'cev condition characterizing congruence distributivity, in \cite{Day.ACOM} A. Day shows a Mal'cev condition characterizing the class of varieties with modular congruence lattices.

These results are examples of a more general theorem obtained independently by Pixley \cite{Pix.LMC} and R. Wille \cite{Wil.K} that can be considered as a foundational result in the field. They proved that if $p \leq q$ is a lattice identity, then the class of varieties whose congruence lattices satisfy $p \leq q$ is the intersection of countably many Mal'cev classes. \cite{Pix.LMC} and \cite{Wil.K} include an algorithm to generate Mal'cev conditions associated with congruence identities.

Furthermore, the class of varieties satisfying a non-trivial idempotent Mal'cev condition (i.e. any  idempotent Mal'cev condition which is not satisfied by any projection algebra) is known to be a Mal'cev class \cite{Tay}. Such class of varieties were characterized by the existence of a {\it Taylor} term, namely an idempotent $n$-ary term $t$ that for every coordinate $i\leq n$ satisfies an identity as 
$$t(x_1,\ldots,x_n)\approx t(y_1,	\ldots, y_n)$$
where $x_1,\ldots,x_n,y_1,\ldots,y_n\in \{x,y\}$, $x_i=x$ and $y_i=y$.

Recently this class of varieties was proven to be a strong Mal'cev class \cite{Miro}, i.e. there exists the weakest strong idempotent Mal'cev condition.

A variety $\mathcal{V}$ is {\it meet-semidistributive} if %the congruence lattice of every algebra in $\mathcal{V}$ is meet-semidistributive. Namely, 
the implication
$$\alpha\wedge \beta=\alpha\wedge\gamma\,\Longrightarrow\, \alpha\wedge\beta=\alpha\wedge (\beta \vee\gamma),$$
holds for every triple of congruences of any algebra in $\mathcal{V}$. It is still unknown if the class of meet-semidistributivity varieties is defined by a strong Mal'cev condition, nevertheless it can be characterized in several different ways \cite{Miro2}. On the other hand, we are going to use the characterization of meet-semidistributive varieties in terms of {\it commutator of congruences} as defined in \cite{comm}. 
\begin{theorem}\cite[Theorem 8.1 items (1), (3), (4)]{shape}\label{KK}
	Let $\mathcal{V}$ be a variety. The following are equivalent:
	\begin{itemize}
		\item[(i)] $\mathcal{V}$ is a congruence meet-semidistributive variety.
		\item[(ii)] No member of $\mathcal{V}$ has a non-trivial abelian congruence.
		\item[(iii)] $[\alpha,\beta] = \alpha\wedge \beta$ for every $\alpha,\beta \in \Con(A)$ and every $A \in \mathcal{V}$.
		%		\item[(iv)] The lattice  
		%\begin{equation}\label{M3}\tag{$M_3$}
		%		 \xymatrixrowsep{0.12in}
		%	\xymatrixcolsep{0.30in}\xymatrix{  & \bullet \ar@{-}[dr] \ar@{-}[d]\ar@{-}[dl] & \\
		% \bullet \ar@{-}[dr] & \bullet \ar@{-}[d] &\bullet \ar@{-}[dl] \\
		%&  \bullet & }
		%\end{equation}
		%		
		%		is not embeddable in $Con(A)$ for any $A\in \mathcal{V}$.
	\end{itemize}
	%		
	%		\comment{Maybe (iv) is NOT needed}
\end{theorem}

An algebra $A$ is said to be: 
\begin{itemize}
	\item[(i)] {\it coherent} if every subalgebra of $A$ which contains a block of a congruence $\alpha\in \Con(A)$ is a union of blocks of $\alpha$. 
	\item[(ii)] {\it Congruence regular} if whenever $[a]_\alpha=[a]_\beta$ for some $a\in A$ and $\alpha,\beta\in \Con(A)$ then $\alpha=\beta$.
	\item[(iii)] {\it Congruence uniform} if the blocks of every congruence $\alpha\in \Con(A)$ have all the same cardinality.
\end{itemize}
A variety $\mathcal{V}$ is {\it coherent} (resp. {\it congruence uniform}, {\it congruence regular}) if all the algebras in $\mathcal{V}$ are coherent (resp. congruence uniform, congruence regular). Because for varieties regularity is equivalent to the condition that no non-zero congruence has a singleton congruence class, every congruence uniform variety is congruence regular. Congruence regularity and coherency are weak Mal'cev classes (see \cite{regular} and \cite{coherent}). On the other hand, it is known that congruence uniformity is not defined by a Mal'cev condition \cite{uniform}. %\comment{Weak class to be defined, non aggiungerei le varie definizioni di condizioni di Mal'cev e term equivalence. Ho controllato e nei vari articoli di questo genere vengono sempre date per scontate }

Some of the most studied Mal'cev classes of varieties are displayed in Figure \ref{Maltsev conditions}. We refer the reader to \cite{UA} 
%\comment{add cit}
for further informations about such classes and to \cite{Bod} for a more exhaustive poset of Mal'cev classes.
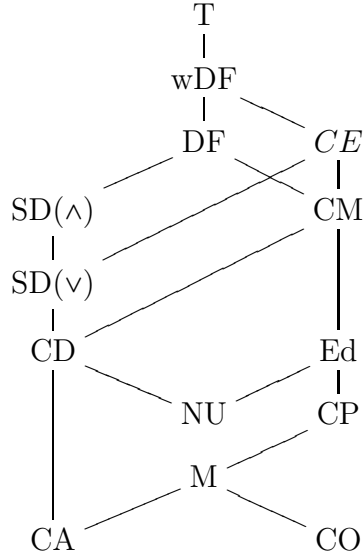
\begin{figure}[!ht]     
	$ \xymatrixrowsep{0.12in}
	\xymatrixcolsep{0.30in}
	\xymatrix{
		& \text{T} \ar@{-}[d]   &   \\
		& \text{wDF} \ar@{-}[dr] \ar@{-}[d]&  &    \\
		& \text{DF} \ar@{-}[dr]\ar@{-}[dl] &    CE\ar@{-}[d] \ar@{-}[ddll]&\\
		\text{SD($\wedge$)} \ar@{-}[d]  & &   \text{CM} \ar@{-}[ddll]\ar@{-}[dd]   \\
		\text{SD($\vee$)} \ar@{-}[d]    & &   \\
		\text{CD}\ar@{-}[ddd] \ar@{-}[dr]  & & \text{Ed} \ar@{-}[d]\ar@{-}[dl]  \\
		&  \text{NU}   &  \text{CP}  \ar@{-}[dl]  \\
		& \text{M}\ar@{-}[dl] \ar@{-}[dr]&      \\
		\text{CA} &   & \text{CO} &   
	}
	$
	\caption{Mal'cev classes: T = Taylor term, wDF = weak difference term, CE = non trivial congruence equation, DF = difference term, CM = congruence modularity, Ed = edge term, CP = congruence permutability, M = Mal'cev term, CO = congruence coherency, SD($\wedge$) = meet semidistributivity, SD($\vee$) = join semidistributivity, CD = congruence distributivity, NU = CD $\bigcap$ Ed = near unanimity term, CA = CD $\bigcap$ M = congruence arithmeticity.}
	\label{Maltsev conditions} 
	
\end{figure}

The main goal of this paper is to investigate Mal'cev conditions for racks and quandles. In particular, this paper is concerned with certain Mal'cev classes of varieties, namely, the varieties having a Taylor term, a Mal'cev term and meet semi-distributive congruence lattices.

Left quasigroups  are rather combinatorial objects, nevertheless Mal'cev classes of varieties of left quasigroups behave in a pretty rigid way. A characterization of Mal'cev varieties of left quasigroups is provided in Theorem \ref{main th}: they are the varieties for which every left quasigroup is connected, (a left quasigroup is connected if the action of its left multiplication group is transitive). Moreover, we show that several Mal'cev conditions are equivalent for varieties of left quasigroups. In particular, all the classes in the interval between the class of Taylor varieties and the class of coherent varieties in Figure \ref{Maltsev conditions} collapse into the strong Mal'cev class of varieties with a Mal'cev term. Moreover, we prove that the weakest non-trivial (not necessarily idempotent) Mal'cev condition for left quasigroups is having a Mal'cev term, and all such varieties are congruence uniform. In Corollary \ref{teo Taylor} we characterize finite Mal'cev idempotent left quasigroups as the superconnected idempotent left quasigroups (i.e. left quasigroups such that all the subalgebras are connected) using a general result given in \cite{SB}.

In Theorem \ref{main th 2} we show that a congruence meet-semidistributive variety of left quasigroups are congruence arithmetic.

As a consequence of our two main theorems, the poset of Mal'cev classes of left quasigroups in figure \ref{Maltsev conditions} turns into the one in Figure \ref{Maltsev conditions for left quasigroup}.

\begin{figure}[!ht]
	$\xymatrixrowsep{0.15in}
	\xymatrixcolsep{0.15in}
	\xymatrix{ 
		& \text{T = CO = M} \ar@{-}[d] & \\
		&\text{NU = SD($\wedge$)}=\text{ CA}	
		%		&\text{SD($\wedge$)}=\text{SD($\vee$)}=\text{CD = NU = CA}
	}
	$ 
	\caption{Mal'cev classes of varieties of left quasigroups.}
	\label{Maltsev conditions for left quasigroup}  
\end{figure}

Then we turn our attention to quandles, i.e. idempotent left distributive left quasigroups. Quandles are of interest since they provide knot invariants \cite{J,Matveev}. The class of quandles used for such topological applications is the class of connected quandles. According to the characterization of Mal'cev varieties of left quasigroups, connectedness is actually a relevant property also algebraically. Some of the contents of the paper are formulated for {\it semimedial} left quasigroups, a class that contain racks and medial left quasigroups \cite{semimedial}.

A characterization of distributive varieties of semimedial left quasigroup is given by the properties of the displacement group in Theorem \ref{distributive semimedial} where we take advantage of the adaptation of the commutator theory in the sense of \cite{comm} developed first for racks in \cite{CP} and then extended to semimedial left quasigroups in \cite{semimedial}.

In Theorem \ref{abel iff finite} we prove that a variety of quandles is distributive if and only if it has no finite models, making use of the characterization of strictly simple and simple abelian quandles \cite{Principal}. We also prove that there is no distributive variety of involutory quandles. The problem of finding an example of non-trivial distributive variety of quandles (resp. left quasigroups) is still open.

Examples of non-trivial Mal'cev varieties of quandles (which members are not just left quasigroup reducts of quasigroups) are provided in Table \ref{non trivial ex}.

\subsection*{Notation and terminology} 

We refer to \cite{UA} for basic concepts of universal algebra.
Let $A$ be an algebra and $t$ be an $n$-ary term. Then we say that $A$ satisfies the identity $t_1(x_1,\ldots, x_n)\approx t_2(x_1,\ldots,x_n)$ if $t_1(a_1,\ldots, a_n)= t_2(a_1,\ldots,a_n)$ for every $a_i\in A$.  %\comment{Add a few words on term equivalence? Secondo me no perché solitamente negli articoli non ho mai visto la definizione formale di term equivalent}

We denote by $\textbf{H}(A)$, $\textbf{S}(A)$ and $\textbf{P}(A)$ respectively the set of homomorphic images, subalgebras and powers of the algebra $A$ and $\mathcal{V}(\mathcal{K})$ denotes the variety generated by the class of algebras $\mathcal{K}$.
We denote by $\Con(A)$ the congruence lattice of $A$, the block of $a\in A$ with respect to a congruence $\alpha$ is denoted by $[a]_\alpha$ (or simply by $[a]$) and the factor algebra by $A/\alpha$. We denote by $1_A=A\times A$ and $0_A=\setof{(a,a)}{a\in A}$ respectively the top and bottom element in the congruence lattice of $A$

Through all the paper, concrete examples of left quasigroups are computed using the software Mace4 \cite{Prover9} and examples of quandles are taken from the library of connected quandles of GAP  \cite{RIG}.

\subsection*{Acknowledgements}
The first author would like to thank Professor David Stanovsk\'y for introducing him to the topic and Professor Paolo Aglian\`o for the useful discussions and remarks about the present paper.

\section{Left quasigroups}

A left quasigroup is a binary algebraic structure $(Q,*,\backslash)$ such that the following
identities hold:
$$x*(x\backslash y)\approx y \approx x\backslash (x*y).$$
Hence, a left quasigroup is a set $Q$ endowed with a binary operation $\ast$  such that the mapping $L_x \colon y \mapsto x \ast y$ is a bijection of $Q$ for every $x \in Q$. The right multiplication mappings $R_x \colon  y \mapsto y \ast x$ need not to be bijections. Clearly the left division is defined by $x/y = L^{-1}_x (y)$, so we usually denote left quasigroups just as a pair $(Q, \ast)$. Nevertheless, if $(Q, \ast)$ is a left quasigroup and $(R, \ast)$ is a binary algebraic structure and $f \colon Q \rightarrow R$ is a homomorphism with respect to $\ast$, the image
of $f$ is not necessarily a left quasigroup. 
We define the {\it left multiplication group of $Q$} as $\lmlt(Q)=\langle L_a,\, a\in Q\rangle$.% \comment{with the composition as multiplication - superfluo? Non so io lo manterrei ma come vuoi per me è uguale}.

Let $\alpha$ be a congruence of a left quasigroup $Q$. The map
$$\lmlt(Q)\longrightarrow \lmlt(Q/\alpha),\quad L_a\mapsto L_{[a]}$$
can be extended to a surjective morphism of groups with kernel denoted by $\lmlt^\alpha$. The {\it displacement group relative to $\alpha$}, denoted by $\dis_\alpha$, is the normal closure in $\lmlt(Q)$ of $\setof{L_a L_b^{-1}}{a\,\alpha\, b}$. In particular, we denote by $\dis(Q)$ the displacement group relative to $1_Q$ and we simpy call it the {\it displacement group of $Q$}. The map defined above clearly restrict and corestrict to the displacement groups of $Q$ and $Q/\alpha$ and we denote by $\dis^\alpha$ the intersection between $\lmlt^\alpha$ and $\dis(Q)$.

\begin{lemma}\label{dis of HSP}
	Let $\mathcal{K}$ be a class of left quasigroups
	%$\mathcal{L}=\setof{\dis(R)}{R\in \mathcal{K}}$, $\setof{\lmlt(R)}{R\in \mathcal{K}})$
	and $Q\in \mathcal{V}(\mathcal{K})$. %Then $\dis(Q)\in \mathcal{V}( \mathcal{D})$ and $\lmlt(Q)\in \mathcal{V}(\mathcal{L})$. 
	Then:
	\begin{itemize}
		\item[(i)] $\dis(Q)\in \mathcal{V}(\setof{\dis(R)}{R\in \mathcal{K}})$.
		\item[(ii)] $\lmlt(Q)\in \mathcal{V}(\setof{\lmlt(R)}{R\in \mathcal{K}})$. 
	\end{itemize}
\end{lemma}

\begin{proof}
	(i) Let $\setof{Q_i}{i\in I}\subseteq \mathcal{K}$. The group $\dis(Q_i/\alpha)\in \textbf{H}(\dis(Q_i))$. Let $S$ be a subalgebra of $Q_i$ and $H=\langle L_a,\, a\in S\rangle$. Then
	$$\dis(S)\cong \langle h L_a L_b^{-1} h^{-1}|_S,\, a,b\in S,\, h\in H\rangle\in \textbf{HS}(\dis(Q_i)).$$
	Let $Q=\prod_{i\in I} Q_i$ and $\alpha_i$ the kernel of the canonical homomorphism onto $Q_i$. Then $\bigcap_{i\in I} \dis^{\alpha_i}=1$ and so we have a canonical embedding 
	\begin{equation*}
		\dis(Q)\hookrightarrow\prod_{i\in I} \dis(Q)/\dis^{\alpha_i}=\prod_{i\in I} \dis(Q_i),
	\end{equation*}
	i.e. $\dis(Q)\in \textbf{SP}(\setof{\dis(Q_i)}{i\in I})$. The same argument can be used for (ii)
\end{proof}

In \cite[Section 1]{semimedial} we introduced the lattice of {\it admissible subgroups} of a left quasigroup $Q$. Given $N\leq \lmlt(Q)$ we have two equivalence relation: 
\begin{itemize}
	\item[(i)] the orbit decomposition with respect to the action of $N$, denoted by $\mathcal{O}_N$.
	\item[(ii)] The equivalence $\c{N}$ defined as
	$$a\, \c{N}\, b\, \text{ if and only if }\, L_a L_b^{-1}\in N.$$
\end{itemize}   

The assignments $\alpha\mapsto \dis_\alpha$ (resp. $\dis^\alpha$) and $N\mapsto \c{N}$ (resp. $\mathcal{O}_N$) are monotone and $\dis_\alpha\leq \dis^\alpha$ (see the characterization of congruences in terms of the properties of subgroups provided in \cite[Lemma 1.5]{semimedial}), wether in general no containment between the equivalences $\c{N}$ and $\mathcal{O}_N$ holds.

We define the lattice of admissible subgroups as
$$\N(Q)=\setof{N\trianglelefteq \lmlt(Q)}{\mathcal{O}_N\subseteq \c{N}}.$$
%\comment{Note that, $N\in \N(Q)$ if and only if $\dis_{\mathcal{O}_N}\leq N$ - USED?}. 
In particular, $\mathcal{O}_N$ is a congruence of $Q$ whenever $N$ is admissible and $\dis_\alpha,\dis^\alpha\in \N(Q)$ for every congruence $\alpha$. The assignments $N \mapsto \mathcal{O}_N$ and $\alpha\mapsto \dis^\alpha$ provide a monotone Galois connection between $\N(Q)$ and the congruence lattice of $Q$ \cite[Theorem 1.10]{semimedial}.

The {\it Cayley kernel} of a left quasigroup $Q$ is the equivalence relation $\lambda_Q$ defined by
$$a\,\lambda_Q\, b \, \text{ if and only if } \, L_a=L_b.$$
Such a relation is not a congruence in general. We say that:
\begin{itemize}
	\item[(i)] $Q$ is a {\it Cayley} left quasigroup if $\lambda_Q$ is a congruence. A class of left quasigroups is Cayley if all its members are Cayley left quasigroups.
	\item[(ii)] $Q$ is {\it faithful} if $\lambda_Q=0_Q$ and $Q$ is {\it superfaithful} if all the subalgebras of $Q$ are faithful.
	\item[(iii)] $Q$ is {\it permutation} if $\lambda_Q=1_Q$, i.e. there exists $f\in \sym{Q}$ such that $a*b=f(b)$ for every $a,b\in Q$. If $f=1$ we say that $Q$ is a {\it projection} left quasigroup (we denote by $\mathcal{P}_n$ the projection left quasigroup of size $n$).
\end{itemize}

According to \cite[Theorem 5.3]{covering_paper}, the {\it strongly abelian} congruences of left quasigroups (in the sense of \cite{TCT}) are exactly those below the Cayley kernel. Equivalently, if $\alpha$ is a congruence of a left quasigroup $Q$, then $\alpha\leq \lambda_Q$ if and only if $\dis_\alpha=1$.

A left quasigroup $Q$ is {\it connected} if its left multiplication group is transitive over $Q$. We say that $Q$ is {\it superconnected} if all the subalgebras of $Q$ are connected. We investigated superconnected left quasigroups in \cite{Super}. %We denote by $\mathcal{P}_2$ the projection left quasigroup as defined in \cite{}

\begin{proposition}\label{pi_0}\cite[Corollary 1.6]{Super}
	A left quasigroup $Q$ is superconnected if and only if $\mathcal{P}_2\notin \textbf{{H}}\textbf{{S}}(Q)$.
\end{proposition}

The property of being (super)connected is also reflected by the properties of congruences.
\begin{lemma}\label{regularity of cong}
	Connected left quasigroups are congruence uniform and congruence regular. 
\end{lemma} 
%\comment{si giusto qui è dimostrato per singole algebre quindi manteniamo and}

\begin{proof}
	Let $Q$ be a connected left quasigroup and assume that $[a]_\alpha=[a]_\beta$ for some $a\in Q$. For every $b\in Q$ there exists $h\in \lmlt(Q)$ with $b=h(a)$. The blocks of congruences are blocks with respect to the action of $\lmlt{(Q)}$, then $$[b]_\alpha=[h(a)]_\alpha=h([a]_\alpha)=h([a]_\beta)=[h(a)]_\beta=[b]_\beta,$$
	and so $\alpha=\beta$. In particular, the mapping $h$ is a bijection between $[a]_\alpha$ and $[b]_\alpha$ for every $\alpha\in \Con(Q)$.
\end{proof}
%In particular connected left quasigroups are congruence regular (i.e. if two congruence share a bock then they are equal).

\begin{lemma}\label{superconnected -> coherent}
	Superconnected left quasigroups are coherent.
\end{lemma}
\begin{proof}
	Let $Q$ be a superconnected left quasigroup, $M\in \mathbb{S}(Q)$ and $\alpha\in \Con(Q)$ with $[a]_\alpha\subseteq M$ for some $a\in M$. For every $b\in M$ there exists $h\in \lmlt(M)$ such that $b=h(a)$. The blocks of $\alpha$ are blocks with respect to the action of $\lmlt(Q)$ and $M$ is a subalgebra, then $h([a]_\alpha)=[b]_\alpha\subseteq M$. Therefore, $M=\bigcup_{b\in M} [b]_\alpha$.
\end{proof}

A {\it quasigroup} is a binary algebra $(Q,*,\backslash,/)$ such that $(Q,*,\backslash)$ is a left quasigroup (the {\it left quasigroup reduct} of $Q$) and $(Q,*,/)$ is a right quasigroup. The left quasigroups obtained as reducts of quasigroups are called {\it latin} (note that congruence and subalgebras of a quasigroup and its left quasigroup reduct might be different due to the different signature considered for the two structures). Latin left quasigroups are superfaithful and connected.

The {\it squaring mapping} for a left quasigroup is the map
$$\s:Q\longrightarrow Q,\quad a\mapsto a*a.$$
We denote the set of {\it idempotent elements of $Q$} by $E(Q)=Fix(\s)=\setof{a\in Q}{a*a=a}$. We say that:
\begin{itemize}
	\item[(i)] $Q$ is {\it idempotent} if $Q=E(Q)$, i.e. the identity $x*x\approx x$ holds in $Q$.
	\item[(ii)] $Q$ is {\it $2$-divisible} if $\s$ is a bijection.
	
	\item[(iii)] $Q$ is {\it $n$-multipotent} is $|\s^n(Q)|=1$ (here $\s^n=\s\circ \s^{n-1}$ denotes the usual composition of maps).
\end{itemize}

\section{Mal'cev classes of left quasigroups}\label{Sec:Mal'cev classes}

In this section we turn our attention to Mal'cev classes of left quasigroups. According to \cite[Theorem 3.13]{shape} a variety with a Taylor term does not contain any strongly abelian congruence, so in particular Taylor varieties of left quasigroup do not contain permutation left quasigroups (if $Q$ is permutation, then $1_Q=\lambda_Q$ is strongly abelian). 
%
%Any \comment{Taylor} variety omits strongly abelian congruences \cite[Theorem 3.13]{shape}. In particular, the left quasigroups in a Cayley Malt'sev variety are faithful. Let us start with a general observation, following by this fact. 

\begin{proposition}\label{dis is transitive}
	Let $\mathcal{V}$ be a Taylor variety of left quasigroups. Then $\dis(Q)$ is transitive on $Q$ for every $Q\in \mathcal{V}$.
\end{proposition}
%\comment{integrate in \ref{main th}?}
\begin{proof}
	Let $Q\in \mathcal{V}$. According to \cite[Corollary 1.9]{semimedial}, $P=Q/\mathcal{O}_{\dis(Q)}$ is a permutation left quasigroup and so $P$ is trivial, i.e. $\dis(Q)$ is transitive on $Q$.
\end{proof}

For left quasigroups, the interval of Mal'cev classes between the class of Taylor varieties and the class of coherent varieties collapses into the class of varieties with a Mal'cev term. 

\begin{theorem}\label{main th}
	Let $\mathcal{V}$ be a variety of left quasigroups. The following are equivalent:
	\begin{itemize}
		\item[(i)] $\mathcal{V}$ has a Mal'cev term.
		\item[(ii)] $\mathcal{V}$ has a Taylor term.
		\item[(iii)] $\mathcal{V}$ satisfies a non-trivial Mal'cev condition.
		\item[(iv)] $\mathcal{P}_2\notin \mathcal{V}$.
		\item[(v)] Every algebra in $\mathcal{V}$ is superconnected.
		\item[(vi)] $\mathcal{V}$ is coherent. 
		
	\end{itemize}
	In particular, every Mal'cev variety of left quasigroup is congruence uniform.
\end{theorem}

\begin{proof}
	The implications (i) $\Rightarrow$ (ii) and (vi) $\Rightarrow$ (i) hold in general as represented  in Figure \ref{Maltsev conditions}, (ii) $\Rightarrow$ (iii) and  (iii) $\Rightarrow$ (iv) clearly hold. 
	
	(iv) $\Rightarrow$ (v) %Let $\mathcal{P}_2\notin \mathcal{V}$. 
	According to Proposition \ref{pi_0}, if $\mathcal{P}_2\notin \mathcal{V}$ then every left quasigroup in $\mathcal{V}$ is connected and then superconnected since $\mathcal{V}$ is closed under taking subalgebras.
	
	(v) $\Rightarrow$ (vi) %Every left quasigroup in $\mathcal{V}$ is superconnected. 
	By Lemma \ref{superconnected -> coherent} every superconnected is coherent, i.e. $\mathcal{V}$ is coherent.
	
	According to Lemma \ref{regularity of cong}, connected left quasigroups are congruence uniform, therefore so it is any Mal'cev variety of left quasigroup.
\end{proof}

\begin{corollary}\label{teo Taylor}
	Let $Q$ be a finite idempotent left quasigroup. Then $\mathcal{V}(Q)$ has a Mal'cev term if and only if $Q$ is superconnected.
\end{corollary}

\begin{proof}
	Let $Q$ be a finite idempotent left quasigroup. According to \cite[Theorem 1.1]{SB}, $\mathcal{V}(Q)$ has Taylor term if and only if $\mathcal{P}_2\notin \textbf{{H}}\textbf{{S}}(Q)$. Thus, $\mathcal{V}(Q)$ has Taylor term if and only if $Q$ is superconnected by Proposition \ref{pi_0}. 
\end{proof}

\begin{proposition}\label{for Cayley}
	Let $\mathcal{V}$ be a Cayley (resp. idempotent) Mal'cev variety of left quasigroups and $Q\in \mathcal{V}$. Then:
	\begin{itemize}
		\item[(i)] every left quasigroups in $\mathcal{V}$ is superfaithful. 
		
		\item[(ii)] The $\dis$ operator is injective and the $\c{}$ operator is surjective and $\alpha=\c{\dis_\alpha}=\c{\dis^{\alpha}}$ for every $\alpha\in \Con(Q)$.
	\end{itemize}
	
\end{proposition}

\begin{proof}
	(i) 
	Idempotent superconnected left quasigroups are superfaithful according to \cite[Lemma 1.9]{Super}, so the claim follows if $\mathcal{V}$ is idempotent. % and then superfaithful since $\mathcal{V}$ is closed under taking subalgebras.
	
	Assume that $\mathcal{V}$ is a Cayley variety. %According to \cite[Theorem 3.13]{shape} a variety with a Taylor term does not contain any strongly abelian congruence. 
	The Cayley kernel is a strongly abelian congruence for Cayley left quasigroups (see \cite[Proposition 5.1]{covering_paper}), therefore the left quasigroups in $\mathcal{V}$ are superfaithful. 
	
	(ii) All the left quasigroups in $\mathcal{V}$ are superfaithful by (i). According to \cite[Proposition 1.6]{semimedial} we have that
	$$\alpha\leq \c{\dis_\alpha}\leq \c{\dis^\alpha}=\alpha.$$
	and so the operator $\c \dis$ is the identity on $\Con(Q)$. 
\end{proof}

Let us turn our attention to congruence distributive varieties of left quasigroups. We have already proved that every Taylor variety of left quasigroups is also Mal'cev. Therefore, the left branch of the poset in Figure \ref{Maltsev conditions} also collapses into the Mal'cev class of distributive varieties.

\begin{theorem}\label{main th 2}
	Let $\mathcal{V}$ be a variety of left quasigroups. The following are equivalent:
	\begin{itemize}
		\item[(i)] $\mathcal{V}$ is congruence meet-semidistributive.
		\item[(ii)] $\mathcal{V}$ is congruence distributive.
		\item[(iii)] $\mathcal{V}$ is congruence arithmetic. 
	\end{itemize}
\end{theorem}

According to Theorems \ref{main th} and \ref{main th 2}, for left quasigroups the poset of Mal'cev classes in Figure \ref{Maltsev conditions} turns into the one in Figure \ref{Maltsev conditions for left quasigroup}.

A term $t(x_1,\ldots, x_n)$ in the language of left quasigroups is a well-formed formal expression using the variables $x_1,\dots,x_n$ and the operations $\{*,\backslash\}$. It is easy to see that the term $t$ is given by 
\begin{equation}
	t(x_1,\ldots x_n)=u(x_1,\ldots, x_n) \bullet r(x_1,	\ldots,x_n)
\end{equation}
where $\bullet\in \{*,\backslash\}$ and $u$ and $r$ are suitable subterms. Let $u$ be a $n$-ary term.  We define
\begin{align*}
	L_{u(x_1,\ldots,x_n)}^0(y)&=y\\ 
	L_{u(x_1,\ldots,x_n)}^{k+1}(y)&=u(x_1,\ldots,x_n)*L_{u(x_1,\ldots,x_n)}^k(y) ,\\
	L_{u(x_1,\ldots,x_n)}^{k-1}(y)&=u(x_1,\ldots,x_n)\backslash L_{u(x_1,\ldots,x_n)}^k(y) ,
	%u(x_1,\ldots,x_n)\backslash y &=L_{u(x_1,\ldots,x_n)}^{-1}(y).
\end{align*}
for $k\in\mathbb{Z}$. Using this notation we have that every term $t$ can be written as
$$t(x_1,\ldots,x_n)=L_{u_1(x_1,\ldots,x_n)}^{k_1}\ldots L_{u_m(x_1,\ldots, x_n)}^{k_m}(x_R)$$
where $u_i$ is a suitable subterm, $k_i=\pm 1$ for $1\leq i\leq m$ and $x_R\in \setof{x_i}{i=1\ldots n}$. We say that $x_r$ is the {\it rightmost variable of $t$}.

%\begin{proposition}\label{taylor_by identities}

Every identity in the language of left quasigroups $t_1 \approx t_2$ has the form
\begin{align*}
	L_{w_1(x_1,\ldots,x_n)}^{k_1}\ldots L_{w_m(x_1,\ldots,x_n)}^{k_m}(x_R) &\approx L_{u_1(y_1,\ldots,y_l)}^{r_1}\ldots L_{u_l(y_1,\ldots,y_l)}^{r_l}(y_R),
\end{align*}
or equivalently, 
\begin{align}\label{identity for Taylor}
	L_{u_l(y_1,\ldots,y_l)}^{-r_l}\ldots L_{u_1(y_1,\ldots,y_l)}^{-r_1}		L_{w_1(x_1,\ldots,x_n)}^{k_1}\ldots L_{w_m(x_1,\ldots,x_n)}^{k_m}(x_R) &\approx y_R.
\end{align}
%	or equivalently as
%	\begin{equation}\label{identity for Taylor}
%L_{r_l(y_1,\ldots,y_l)}^{-u_l}\ldots L_{r_1(y_1,\ldots,y_l)}^{-u_1} L_{w_1(x_1,\ldots,x_n)}^{k_1}\ldots L_{w_m(x_1,\ldots,x_n)}^{k_m}(x_i)\approx y_j.
%\end{equation}
The projection left quasigroup $\mathcal{P}_2$ satisfies \eqref{identity for Taylor} if and only if $x_R=y_R$. So a variety of left quasigroups $\mathcal{V}$ has a Mal'cev term if and only if it satisfies an identity as in \eqref{identity for Taylor} with $x_R \not= y_R$.
%	of the following form
%		\begin{equation}\label{identity for Taylor1}
%		t(x_1,\ldots, x_n)=	L_{w_1(x_1,\ldots,x_n)}^{k_1}\ldots L_{w_m(x_1,\ldots,x_n)}^{k_m}(y)\approx x.
%		\end{equation}
%	where $n,m \in\mathbb{N}$, $k_i\in \mathbb{Z}$ and $s_i$ are suitable subterms of $t$ for every $1\leq i\leq  n$.

%\begin{remark}
Note that, an identity as in \eqref{identity for Taylor} might have just the trivial model. For instance if $\mathcal{V}$ is a variety of idempotent left quasigroups satisfying such an identity and the variable $x$ does not appear in the left handside then $\mathcal{V}$ is trivial. Indeed, identifying all the variables $x_1,\ldots, x_n,y $ we have $L_y^{k_1+\ldots+ k_m}(y)=y\approx x$.

\begin{example}	\label{latin Maltsev}
	
	A variety axiomatized by some identities as in \eqref{identity for Taylor} might be made up of latin left quasigroups. For instance, Mal'cev varieties of left quasigroups are provided by varieties of quasigroup in which every member is {\it term equivalent} to its left quasigroup reduct. This is the case of the following examples (for an example of a Mal'cev variety of latin left quasigroups not arising from quasigroups see Proposition \ref{unipotent are latin}).
	\begin{itemize}
		\item[(i)] The variety of commutative left quasigroups defined by the identity
		$$x*y=y*x.$$
		%Note that \comment{$M_n\subseteq M_{n+1}$ - check by Prover9} and $M_2$ is the variety of commutative left quasigroups. 
		%\comment{$R_{(\ldots(x_3(x_2\cdot x_1))\ldots)}=L_{(\ldots ((x_1\cdot x_2)x_3) \ldots)}$}
		
		%, i.e. the quasigroups satisfying  
		%		\begin{displaymath}
		%		x\ast y \approx y\ast x.
		%		\end{displaymath}
		%\comment{	The left quasigroups satisfying the identity 
		%	\begin{equation}\label{star}
		%	x*(y*z)\approx (z*y)*x
		%	\end{equation}	
		%are latin. Indeed \eqref{star} is satisfied by a left quasigroup $Q$ if and only if $R_{y*z}=L_{z*y}$ for every $y,z\in Q$. So \eqref{star} implies that $R_y=L_{(y\backslash y)*y} $ is bijective for every $y\in Q$. 
		%
		%The identity\eqref{star} states that $t(x,y,z)=(z*y)*x$ is an LTT term, but the righmost variable is changing. Call such LQGs $n$-mirror or something as generalization of commutative? Note that $n$-mirror $\Rightarrow$ $n+1$-mirror. 
		%		
		
		%		The subvariety of commutative racks, called the variety of symmetric quandles in \cite{Montoli}, is actually the variety of commutative distributive quasigroups. 
		\item[(ii)] Let $n\in \mathbb{N}$. The variety of left quasigroups satisfying the identity
		\begin{displaymath}
			(\ldots((x\ast \underbrace{y)\ast y)\ldots)\ast y}_{n} \approx x.
		\end{displaymath}
		%		
		%		\item[(iii)] \comment{Remove?} The variety of left quasigroups satisfying the identity
		%		\begin{displaymath}
		%		(x\ast  y)\ast x \approx y\approx x\ast  (y\ast x) .
		%		\end{displaymath}
		\item[(iii)] The variety of paramedial left quasigroups, identified by the identity
		$$(x* y)*(z* t)=(t* y)*(z* x).$$
	\end{itemize}
	
\end{example}

\begin{example}\label{constant power LQG} Mal'cev varieties of left quasigroups are not limited to varieties of latin left quasigroups, as witnessed by the following examples.
	\begin{itemize}
		
		\item[(i)] Let $\mathcal{V}_n$ be the variety of left quasigroups satisfying $L_x^n(x)\approx L_y^n(y)$ where $n\in \mathbb{Z}$. Then
		\begin{equation*}\label{m}
			m(x,y,z)=L_x^{-n}L_y^{n}(z)
		\end{equation*}
		is a Mal'cev term. Let $n>0$, $Q$ be a set and $e$ be a fixed element in $Q$. We define $L_e=1$ and $L_a$ to be any cycle $(a, \ldots, e)$ of length $n$ for every $a\in Q$ (if $n<0$ we define $L_a^{-1}$ in the same way). Then $(Q,*)\in \mathcal{V}_n$. %\comment{Add the table of the $4$-element example?}
		%
		%
		%A construction for such left quasigroups over a set $Q$ and for $n\in \mathbb{Z}$ which are not latin is the following: let $e$ be a fixed element in $Q$, we define $L_e=1$ and $L_a$ to be any cycle $(a, \ldots, e)$ of length $n$ for every $a\in Q$ if $n>0$. If $n<0$ we define $L_a^{-1}$ in the same way (see Example \ref{first example} for $n=1$ and $|Q|=4$).
		
		\item[(ii)] The variety of $n$-multipotent left quasigroups is axiomatized by the identity
		$$\s^n(x)=L_{\s^{n-1}(x)}L_{\s^{n-2}(x)}\ldots L_{\s(x)}L_x(x)\approx L_{\s^{n-1}(y)}L_{\s^{n-2}(y)}\ldots L_{\s(y)}L_y(y)=\s^n(y).$$
		%\comment{According to \cite[Section 2]{Maltsev_paper}} every variety of multipotent left quasigroup has a Malt'sev term. In particular, every multipotent left quasigroup is superconnected. 
		A Mal'cev term for $n$-multipotent left quasigroups is $$m(x,y,z)=\left(L_{\s^{n-2}(x)}\ldots L_{\s(x)}L_x\right)^{-1}  L_{\s^{n-2}(y)}\ldots L_{\s(y)}L_y(z).$$
	\end{itemize}
\end{example}

%
%\begin{example}\label{nonexample}
%The following are examples of varieties of left quasigroups with no Mal'cev term.
%
%\begin{itemize}
%\item[(i)] The variety of {\it semimedial} left quasigroups given by the identity
%$$x*(y*z)\approx (x*x)*(y*z)$$
%\end{itemize}
%
%
%\end{example}

\begin{example}\label{nonexmple}
	Let $\mathfrak{G}$ be a variety of groups. We denote the class of left quasigroups such that the left multiplication group (resp. displacement group) belongs to $\mathfrak{G}$ by $L(\mathfrak{G})$ (resp. $D(\mathfrak{G})$). According to Lemma \ref{dis of HSP} such classes are varieties. Since $\lmlt(\mathcal{P}_2)=\dis(\mathcal{P}_2)=1$ then $\mathcal{P}_2$ belongs to $L(\mathfrak{G})$ and to $D(\mathfrak{G})$ and so they have no Mal'cev term.
\end{example}

\section{Semimedial left quasigroups}\label{Sec:semimedial}

{\it Semimedial} left quasigroups are defined by the semimedial law:
$$(x*y)*(x*z)\approx (x*x)*(y*z).$$
The projection left quasigroup $\mathcal{P}_2$ satisfies the semimedial law and so the whole variety of semimedial left quasigroup in not Mal'cev.

A relevant subvariety of $2$-divisible semimedial left quasigroups is the variety of {\it racks}, axiomatized by the identity
$$x*(y*z)=(x*y)*(x*z).$$
Idempotent semimedial left quasigroups are racks and they are called {\it quandles}. If $Q$ is semimedial then the squaring map $\s$ is a homomorphism and so if $h=L_{a_1}^{k_1}\ldots L_{a_n}^{k_n}\in \lmlt(Q)$ we have
$$\s h=\underbrace{L_{\s(a_1)}^{k_1}\ldots L_{\s(a_n)}^{k_n}}_{=h^\s} \s$$
and the subset $E(Q)=\setof{a\in Q}{a*a=a}$ is a subquandle of $Q$. {\it Medial} left quasigroups, i.e. those for which
$$(x*y)*(z*t)=(x*z)*(y*t)$$
holds are also semimedial. 

For a semimedial left quasigroup $Q$, the admissible subgroup are 
$$\N(Q)=\setof{N\trianglelefteq \lmlt(Q)}{N^\s\leq N}$$
where $N^\s=\setof{h^\s}{h\in N}$. Note that $[g,h]^\s=[g^\s,h^\s]$ for every $g,h\in \lmlt(Q)$. Thus, if $N\in \N(Q)$ then $[\lmlt(Q),N]\in \N(Q)$ (see \cite[Lemma 3.1]{semimedial}).

The relation $\c{N}$ is a congruence for every admissible subgroup $N$ and the assignments $\alpha\mapsto \dis_\alpha$ and $N\mapsto \c{N}$ provide a second monotone Galois connection between the lattice of congruences and the admissible subgroups \cite[Theorem 3.5]{semimedial}. Such a Galois connection is also well-behaved with respect to the commutator of congruences. Indeed, in a Mal'cev variety the commutator of congruences in the sense of \cite{comm} is completely determined by such Galois connection.

\begin{lemma}\label{commutator for superconnected}
	Let $\mathcal{V}$ be a Mal'cev variety of semimedial left quasigroups and $Q\in \mathcal{V}$. Then 
	$$[\alpha,\beta]=\c{[\dis_\alpha,\dis_\beta]}$$ for every $\alpha,\beta\in \mathrm{Con}(Q)$.
\end{lemma}

\begin{proof}
	The variety $\mathcal{V}$ is Cayley (\cite[Proposition 3.6]{semimedial}), and so the left quasigroups in it are superfaithful by Proposition \ref{for Cayley}(i). Therefore we can apply directly \cite[Proposition 3.10]{semimedial}
\end{proof}

Let us show that unipotent semimedial left quasigroups are latin, providing an example of variety of latin left quasigroups that is not term equivalent to a variety of quasigroups. Recall that a group $G$ acting on a set $Q$ is {\it regular} if for every $a,b\in Q$ there exists a unique $g\in G$ such that $b=g\cdot a$. Equivalently the action is transitive and the pointwise stabilizers are trivial.

\begin{proposition}\label{unipotent are latin}
	Let $Q$ be a unipotent semimedial left quasigroup and $\s(Q)=\{e\}$. Then:
	\begin{itemize}
		\item[(i)] the group $\dis(Q)$ is regular and $\dis(Q)=\setof{ L_a L_e^{-1}}{a\in Q}$.
		\item[(ii)] $Q$ is latin.
	\end{itemize}
\end{proposition}
%\comment{maybe more details in the proof? Non saprei}
\begin{proof}

	%If $h=L_{x_1}^{k_1}\ldots L_{x_n}^{k_n}=h^\s=L_e^{k_1+\ldots +k_n}$. Therefore $\lmlt(Q)\cap \aut{Q}=\langle L_e\rangle$ and $\dis(Q)\cap \aut{Q}=1$. 
	%Since $L_e\in \aut{Q}$ then $\dis(Q)\cap \langle L_e\rangle=1$ and so $\lmlt(Q)=\dis(Q)\rtimes \langle L_e\rangle$.
	(i) Let $h =L_{a_1}^{k_1}\ldots L_{a_n}^{k_n}\in \dis(Q)$. According to \cite[Lemma 1.4]{semimedial} $k_1+\ldots+k_n=0$ and so $h^\s=L_{\s(a_1)}^{k_1}\ldots L_{\s(a_n)}^{k_n}=L_e^{k_1+\ldots +k_n}=1$. If $h\in \dis(Q)_a$, then $L_a=L_{h(a)}=h^s L_a h^{-1}=L_a h^{-1}$, i.e. $h=1$ and so $\dis(Q)$ is regular. On the other hand, $e=(e\ldiv a)*(e\ldiv a)=L_{e\ldiv a} L_e^{-1}(a)$, and so we have $\dis(Q)=\setof{L_a L_e^{-1}}{a\in Q}$.

	(ii) Let $a,b\in Q$. According to (i) $\dis(Q)=\setof{L_c L_e^{-1}}{c\in Q}$ and it is regular. Thus, there exists a unique $c$ such that
	$$a=L_c L_e^{-1}(b)=c*(e\ldiv b)$$
	and so the right multiplication $R_{e\ldiv b}$ is biejctive for every $b\in Q$.
\end{proof}

\subsection{Congruence distributive varieties of semimedial left quasigroups}

According to Theorem \ref{main th 2} we have that congruence meet-semidistributive varieties of left quasigroups are congruence distributive. For sememedial left quasigroups congruence distributivity is determined by the properties of the relative displacement groups and of the admissible subgroups.

\begin{proposition}\label{distributive semimedial}
	Let $\mathcal{V}$ be a variety of semimedial left quasigroup. The following are equivalent: 
	\begin{itemize}
		\item[(i)] $\mathcal{V}$ is distributive.
		
		\item[(ii)] $\dis_\alpha=[\dis_\alpha,\dis_\alpha]$ for every $Q\in\mathcal{V}$ and $\alpha\in \mathrm{Con}(Q)$. 
		
		\item[(iii)] If $N\in\N(Q)$ is solvable then $N=1$ for every $Q\in\mathcal{V}$.

	\end{itemize}
\end{proposition} 

\begin{proof}
	It is enough to prove the equivalence for meet-semidistributive varieties thanks to Theorem \ref{main th 2}.
	
	Let $Q\in\mathcal{V}$ and $\alpha\in \mathrm{Con}(Q)$. By Proposition \ref{commutator for superconnected} we have 
	$$\dis_{[\alpha,\alpha]}=\dis_{\c{[\dis_\alpha,\dis_\alpha]}}\leq [\dis_\alpha,\dis_\alpha]\leq \dis_\alpha.$$
	%and the $\dis$ operator in injective.
	
	(i) $\Rightarrow$ (ii) By Theorem \ref{KK} we have $[\alpha,\alpha]=\alpha$ and so $\dis_\alpha=\dis_{[\alpha,\alpha]}=[\dis_\alpha,\dis_\alpha]$.
	%
	%(ii) $\Rightarrow$ (i) If $\dis_\alpha=\dis_{[\alpha,\alpha]}=[\dis_\alpha,\dis_\alpha]$, then $\alpha=[\alpha,\alpha]$ since the $\dis$ operator is injective. In particular there is no non-trivial abelian congruence in $\mathcal{V}$, thus $\mathcal{V}$ is meet-semidistributive according to Theorem \ref{KK}.
	
	(ii) $\Rightarrow$ (iii) Let $N\in \N(Q)$ be solvable of length $n$ and let $D$ be the non-trivial $(n-1)$th element of the derived series of $N$. So $D$ is abelian and it is in $\N(Q)$. Hence, according to \cite[Lemma 2.6]{semimedial}, $\beta=\mathcal{O}_{D}$ is a non-trivial abelian congruence of $Q$. Therefore $\dis_\beta$ is abelian and we have $\dis_\beta=[\dis_\beta,\dis_\beta]=1$. Hence, $\beta\leq \lambda_Q=0_Q$, contradiction.
	
	(iii) $\Rightarrow$ (i) If $\alpha$ is abelian then $\dis_\alpha$ is abelian \cite[Corollary 5.4]{CP}. Hence $\dis_\alpha=[\dis_\alpha,\dis_\alpha]=1$, i.e. $\alpha\leq\lambda_Q=0_Q$. 
\end{proof}

If $Q$ is a $2$-divisible semimedial left quasigroup then 
$$\N(Q)=\setof{N\trianglelefteq \lmlt(Q)}{\s N\s^{-1}\leq N}$$
since $\s$ is bijective. In particular, $Z(N)$ is a characteristic subgroup of $N$, and so it is normal in $\lmlt(Q)$ and $\s Z(N) \s^{-1}\leq Z(N)$. Thus, $Z(N)\in \N(Q)$.
\begin{proposition}\label{2 div and distrib}
	Let $\mathcal{V}$ be a variety of $2$-divisible semimedial left quasigroup. The following are equivalent
	\begin{itemize}
		\item[(i)] $\mathcal{V}$ is distributive 
		\item[(ii)] $Z(N)=1$ for every $Q\in \mathcal{V}$ and every $N\in \N(Q)$.
	\end{itemize}
\end{proposition} 
\begin{proof}
	
	We are using the characterization of distributive varieties given in Proposition \ref{distributive semimedial}(iii).
	
	(i) $\Rightarrow$ (ii) If $N\in\N(Q)$, then $Z(N) \in \N(Q)$ is solvable and so $Z(N)=1$.
	
	(ii) $\Rightarrow$ (i) If $Z(N)=1$ for every $N\in \N(Q)$ then there are no abelian subgroups in $\N(Q)$. Since $[N,N]\in \N(Q)$ for every $N\in \N(Q)$ then there are no solvable subgroup in $\N(Q)$.
\end{proof}

\begin{corollary}\label{no medial in meet semi 1}
	Let $\mathcal{V}$ be a distributive variety of semimedial left quasigroup. Then:
	\begin{itemize}
		\item[(i)] $\mathcal{V}$ does not contain any non-trivial medial left quasigroup.
		\item[(ii)] $\mathcal{V}$ does not contain any non-trivial finite $2$-divisible latin left quasigroup.
		%\item[(iii)] If $E(Q)$ is finite then $|E(Q)|=1$ for every $Q\in \mathcal{V}$.
	\end{itemize}
	In particular, there is no distributive variety of medial left quasigroups. 
\end{corollary}

\begin{proof}
	The variety $\mathcal{V}$ omits solvable algebras. Medial left quasigroups are nilpotent \cite[Corollary 4.4]{semimedial} and finite $2$-divisible latin semimedial left quasigroups are solvable \cite[Corollary 3.20]{semimedial}.
	%
	%\item[(iii)] According to \ref{no SD} if $E(Q)$ is finite then $\mathcal{V}(E(Q))$ contains an abelian algebra. \qedhere
\end{proof}
%
%
%\begin{corollary}
%There is no distributive variety of medial left quasigroups. 
%\end{corollary}
%
%The following problem is still open. \comment{Move to intro?}
%\begin{problem}
%Provide an example of a distributive variety of semimedial left quasigroups.
%%	Do \comment{distributive} varieties of quandles exist?
%\end{problem}

\subsection{Mal'cev varieties of quandles}
%
%\text{}
%
%\comment{Reorder? Need to add this?
%\begin{itemize}
%\item $\dis(Q)=\langle L_a^{-1} L_b,\, a,b\in Q\rangle$.
%\item $\N(Q)=NSub(\lmlt(Q))$.
%\end{itemize}}

In this Section we focus on quandles. A remarkable construction of quandles is the following.

\begin{example}\cite{J}\label{coset quandle}
	Let $G$ be a group, $f\in  \aut{G}$ and $H \leq Fix(f)=\setof{a\in G}{f(a)=a}$. Let $G/H$ be the set of left cosets of $H$ and the multiplication defined by
	\begin{displaymath}
		aH\ast bH=af(a^{-1}b)H.
	\end{displaymath}
	Then $\Q(G,H,f)=(G/H,\ast,\backslash) $ is a quandle, called a \emph{coset} quandle. A coset quandle $\Q(G,H,f)$ is called \emph{principal} if $H=1$ and is such case it is denoted by $\Q(G,f)$. A principal quandle is called \emph{affine} if $G$ is abelian and in such case it is denoted by $\aff(G,f)$.
\end{example}

Connected quandles can be represented as coset quandles over their displacement group.
\begin{proposition}\cite[Theorem 4.1]{hsv}
	Let $Q$ be a connected quandle $Q$. Then $Q$ is isomorphic to $\mathcal{Q}(\dis(Q),\dis(Q)_a,\widehat{L_a})$ for every $a\in Q$ ($\widehat{L}_a$ is the inner automorphism of $L_a$ restricted and corestricted to $\dis(Q)$).
%	
%	where 
%$$\widehat{L}_a:\dis(Q)\longrightarrow \dis(Q),\quad h\mapsto L_a h L_a^{-1}$$	
%is an automorphism of $\dis(Q)$.	
\end{proposition} 
%\comment{Darei una spiegazione del significato di questa cosa $Q\cong \mathcal{Q}(\dis(Q),\dis(Q)_a,\widehat{L_a})$ }

The class of latin quandles is not a subvariety of the variety of quandles. Indeed the non-connected quandle $\aff(\mathbb{Z},-1)$ embeds into the latin quandle $\aff(\mathbb{Q},-1)$. On the other hand, the class of principal quandles of a Mal'cev variety is a subvariety.
\begin{theorem}\label{Malstev principal is a variety}
	The class of principal quandles of a Mal'cev variety $\mathcal{V}$ is a subvariety of $\mathcal{V}$.
\end{theorem}

\begin{proof}
	The product of principal quandles is principal \cite[Corollary 2.3]{Principal}. By virtue of \cite[Proposition 2.11]{Super} subquandles and factors of principal Mal'cev quandles are principal. Hence the class of principal quandles of $\mathcal{V}$ is a subvariety.
\end{proof}

%\begin{example}
The smallest examples of non-latin superconnected quandles are {\tt SmallQuandle}(28,i) for $i=3,4,5,6$ in the \cite{RIG} library of GAP. 
The identities in Table \ref{non trivial ex} provides Mal'cev varieties of quandles that contain such minimal examples.

\begin{table}[ht]	\caption{Examples of Mal'cev varieties of quandles}	 \label{non trivial ex}	
	\begin{tabular}{ | c | c |}
		\hline
		Identity & Witness in the RIG library \\
		\hline 
		\quad $ L_x L_y^2 L_x L_y L_x^2 L_y L_x L_y^2(x) \approx y$\quad & {\tt SmallQuandle}(28,3)\\
%		\quad \comment{$ (x(y(y(x(y(x(x(y(x(y(yx)))))))))) \approx y$}\quad & {\tt SmallQuandle}(28,3)\\
		\quad$L_x^2 L_y L_x L_y^2  L_x L_y L_x^2 L_y^2(x) \approx y$ \quad& {\tt SmallQuandle}(28,4)\\
		\quad$L_x L_y^2 L_x L_y L_x^2 L_y L_x L_y^2(x)\approx y$\quad& {\tt SmallQuandle}(28,5)\\
		\quad$ L_x L_y^2L_x L_y L_x^2 L_y L_x L_y^2(x)\approx y$\quad& {\tt SmallQuandle}(28,6) \\
		\hline
	\end{tabular}
\end{table}
%	\end{example}

Distributive varieties of quandles have the following characterization.

\begin{theorem}\label{abel iff finite}
	Let $\mathcal{V}$ be a variety of quandles. The following are equivalent:
	\begin{itemize}
		\item[(i)] $\mathcal{V}$ contains an abelian quandle.
		
		\item[(ii)] $\mathcal{V}$ has a finite model.
	\end{itemize}
	%In particular, $\mathcal{V}$ is meet-semidistributive if and only if it has no non-trivial finite models.
	In particular, $\mathcal{V}$ is distributive if and only if $\mathcal{V}$ has no finite model.
\end{theorem}

\begin{proof}
	(i) $\Rightarrow$ (ii) According to \cite[Theorem 3.21]{Principal} simple abelian quandles are finite. If $Q\in\mathcal{V}$ is an abelian quandle. According to the main result of \cite{Magari}, $\mathcal{V}(Q)\subseteq \mathcal{V}$ contains a simple abelian quandle which is finite.
	
	(ii) $\Rightarrow$ (i) Let assume that $\mathcal{V}$ contains a finite quandle $Q$. According to \cite[Theorem 4.7]{Principal}, the minimal subquandles of $Q$ with respect to inclusion are abelian. 
	
	The variety $\mathcal{V}$ is idempotent, and so it contains an abelian congruence if and only if it contains an abelian algebra. Thus, the last claim follows.%Hence we can apply Lemma \ref{no SD}.
\end{proof}
%
%\comment{Add blabla}
%
%\begin{theorem}\label{no SD}
%Let $\mathcal{V}$ be a variety of quandles. The following are equivalent:
%\begin{itemize}
%\item[(i)] $\mathcal{V}$ is distributive.
%
%\item[(ii)] $\mathcal{V}$ has no finite model.
%
%
%
%%\item[(iii)] $[\dis_\alpha,\dis_\alpha]=\dis_\alpha$ and $Z(N)=1$ for every $Q\in \mathcal{V}$ and every $\alpha\in Con(Q)$ and $N\in \N(Q)$.
%
%\end{itemize}
%%In particular, $\mathcal{V}$ is distributive if and only if $\mathcal{V}$ has no finite model.
%%In particular, there is no locally finite distributive variety of quandles.
%\end{theorem}
%
%\begin{proof}
%The variety $\mathcal{V}$ is idempotent, and so it contains an abelian congruence if and only if it contains an abelian algebra. Hence we can apply Lemma \ref{no SD}.
%\end{proof}
%%

\begin{corollary}\label{no medial in meet semi}
	Let $\mathcal{V}$ be a distributive variety of semimedial left quasigroup and $Q\in \mathcal{V}$. If $E(Q)$ is finite then $|E(Q)|=1$. 
\end{corollary}

\begin{proof}
	According to Theorem \ref{abel iff finite} if $E(Q)$ is finite then $\mathcal{V}(E(Q))$ contains an abelian algebra.
\end{proof}

{\it Involutory} quandles are the quandles that satisfy the identity $x(xy)\approx y$. A direct conseguence of the contents of \cite[Section 3]{Super} is that connected involutory quandles on two generators are finite, so we have the following Corollary of Theorem \ref{abel iff finite}.

\begin{corollary}\label{no involutory distributive}
	There is no distributive variety of involutory quandles.
\end{corollary}

\bibliographystyle{plain}
\bibliography{references}

\def\cprime{$'$} \def\cprime{$'$}
\begin{thebibliography}{10}

\bibitem{SB}
Libor Barto, Marcin Kozik, and David Stanovsk\'{y}.
\newblock Mal'tsev conditions, lack of absorption, and solvability.
\newblock {\em Algebra Universalis}, 74(1-2):185--206, 2015.

\bibitem{UA}
Clifford Bergman.
\newblock {\em Universal algebra}, volume 301 of {\em Pure and Applied
  Mathematics (Boca Raton)}.
\newblock CRC Press, Boca Raton, FL, 2012.
\newblock Fundamentals and selected topics.

\bibitem{Bod}
M.~Bodirski.
\newblock Mal'cev condition figure.
\newblock {\em https://www.math.tu-dresden.de/~bodirsky/Maltsev-Conditions/|},
  2021.

\bibitem{Principal}
Marco Bonatto.
\newblock Principal and doubly homogeneous quandles.
\newblock {\em Monatshefte f{\"u}r Mathematik}, 191(4):691--717, 2020.

\bibitem{semimedial}
Marco Bonatto.
\newblock Medial and semimedial left quasigroups.
\newblock {\em Journal of Algebra}, 2021.

\bibitem{Super}
Marco {Bonatto}.
\newblock {Superconnected left quasigroups and involutory quandles}.
\newblock {\em arXiv e-prints}, 2021.

\bibitem{covering_paper}
Marco {Bonatto} and David {Stanovsk{\'y}}.
\newblock {A Universal algebraic approach to rack coverings}.
\newblock {\em arXiv e-prints}, page arXiv:1910.09317, Oct 2019.

\bibitem{CP}
Marco {Bonatto} and David {Stanovsk{\'y}}.
\newblock {Commutator theory for racks and quandles}.
\newblock {\em J. Math. Soc. Japan}, 73:41--75, 2021.

\bibitem{regular}
B.~Csakany.
\newblock Characterizations of regular varieties.
\newblock {\em Acta Sci. Math. (Szeged)}, 31:187--189, 1970.

\bibitem{Day.ACOM}
A.~Day.
\newblock A characterization of modularity for congruence lattices of algebras.
\newblock {\em Canad. Math. Bull.}, 12:167--173, 1969.

\bibitem{comm}
Ralph Freese and Ralph McKenzie.
\newblock {\em Commutator theory for congruence modular varieties}, volume 125
  of {\em London Mathematical Society Lecture Note Series}.
\newblock Cambridge University Press, Cambridge, 1987.

\bibitem{coherent}
D.~Geiger.
\newblock Coherent algebras.
\newblock {\em Notices Amer. Math. Soc.}, 21:0, 1974.

\bibitem{TCT}
McKenzie~R. Hobby~D.
\newblock {\em The structure of finite algebras}, volume~76 of {\em
  Contemporary Mathematics}.
\newblock American Mathematical Society, 1988.

\bibitem{hsv}
Alexander Hulpke, David Stanovsk\'{y}, and Petr Vojt\v{e}chovsk\'{y}.
\newblock Connected quandles and transitive groups.
\newblock {\em J. Pure Appl. Algebra}, 220(2):735--758, 2016.

\bibitem{Jon.AWCL}
Bjarni Jonnson.
\newblock Algebras whose congruence lattices are distributive.
\newblock {\em Mathematica Scandinavica}, 21:110--121, Dec. 1967.

\bibitem{J}
David Joyce.
\newblock A classifying invariant of knots, the knot quandle.
\newblock {\em J. Pure Appl. Algebra}, 23(1):37--65, 1982.

\bibitem{shape}
Keith~A. Kearnes and Emil~W. Kiss.
\newblock The shape of congruence lattices.
\newblock {\em Mem. Amer. Math. Soc.}, 222(1046):viii+169, 2013.

\bibitem{Magari}
R~Magari.
\newblock Una dimostrazione del fatto che ogni variet\'a ammette algebre
  semplici.
\newblock {\em Ann. Univ. Ferrara}, 14:1--4, 1969.

\bibitem{Malt}
A.~I. Mal\cprime~cev.
\newblock On the general theory of algebraic systems.
\newblock {\em Amer. Math. Soc. Transl. (2)}, 27:125--142, 1963.

\bibitem{RIG}
Mat{\'{\i}}as Gra{\~n}a and Leandro Vendramin.
\newblock {\em {Rig, a GAP package for racks, quandles and Nichols algebras}}.

\bibitem{Matveev}
S.~V. Matveev.
\newblock Distributive groupoids in knot theory.
\newblock {\em Mat. Sb. (N.S.)}, 119(161)(1):78--88, 160, 1982.

\bibitem{Prover9}
W.~McCune.
\newblock Prover9 and mace4.
\newblock \verb|http://www.cs.unm.edu/~mccune/prover9/|, 2005--2010.

\bibitem{Miro}
Miroslav Ol\v{s}\'{a}k.
\newblock The weakest nontrivial idempotent equations.
\newblock {\em Bull. Lond. Math. Soc.}, 49(6):1028--1047, 2017.

\bibitem{Miro2}
Miroslav Ol\v{s}\'{a}k.
\newblock {Maltsev conditions for general congruence meet-semidistributive
  algebras}.
\newblock {\em arXiv e-prints}, page arXiv:1810.03178, Oct 2018.

\bibitem{Pix.DAPO}
A.F. Pixley.
\newblock Distributivity and permutability of congruence relations in
  equational classes of algebras.
\newblock {\em Proc. Amer. Math. Soc.}, 14:105--109, 1963.

\bibitem{Pix.LMC}
A.F. Pixley.
\newblock Local mal'cev conditions.
\newblock {\em Canad. Math. Bull.}, 15:559--1568, 1972.

\bibitem{uniform}
Walter Taylor.
\newblock Uniformity of congruences.
\newblock {\em Algebra Universalis}, 4:342--360, 1974.

\bibitem{Tay}
Walter Taylor.
\newblock Varieties obeying homotopy laws.
\newblock {\em Canad. J. Math.}, 29(3):498--527, 1977.

\bibitem{Wil.K}
R.~Willie.
\newblock Kongruenzklassengeometrien.
\newblock {\em Lecture Notes in Mathematics, Springer-Verlag, Berlin-New York},
  113, 1970.

\end{thebibliography}

\end{document}